\documentclass[12pt]{amsart}
\usepackage[utf8]{inputenc}
\usepackage{amsthm}
\usepackage{amsmath}
\usepackage{amssymb}
\usepackage{comment}
\usepackage{tikz}
\usepackage[margin=1.2in]{geometry}

\usepackage[T1]{fontenc}
\usepackage{enumitem}
\usepackage{setspace}
\usepackage{microtype}
\usepackage{cite}
\usetikzlibrary{decorations.pathreplacing}

\allowdisplaybreaks

\usepackage[colorlinks=true, pdfstartview=FitV, linkcolor=blue, citecolor=blue, urlcolor=blue]{hyperref}

\newtheorem{theorem}{Theorem}[section]
\newtheorem{definition}[theorem]{Definition}
\newtheorem{lemma}[theorem]{Lemma}
\newtheorem{prop}[theorem]{Proposition}
\newtheorem{corollary}[theorem]{Corollary}
\newtheorem*{remark}{Remark}

\newtheorem{conj}[theorem]{Conjecture}

\providecommand{\wh}{\widehat}

\providecommand{\C}{\mathbb C}

\providecommand{\R}{\mathbb R}

\providecommand{\M}{\mathcal{M}}

\providecommand{\J}{\mathcal{J}}

\providecommand{\supp}{\text{supp }}

\providecommand{\sQ}{\mathcal{Q}}
\providecommand{\T}{\mathcal{T}}
\providecommand{\J}{\mathcal{J}}
\providecommand{\L}{\mathcal{L}}

\providecommand{\bS}{\mathbf{S}}

\title[Strip Decoupling]{Strip Decoupling Inequalities for the Paraboloid in $\R^3$}

\author{Jacob Glidewell}
\address{Dept. of Mathematics \\
Indiana University \\
 Bloomington, IN, 47405-7106, USA}
\email{jaglide@iu.edu}

\begin{document}

\begin{abstract}
    Large-cap or large-ensemble decoupling inequalities were initially studied by Demeter\cite{demeter2026BeyondCanonical}. These are intimately connected to the reverse square-function conjecture and restriction conjecture in Fourier restriction theory. By strips, we mean a partition of the $R^{-1}$-neighborhood of the paraboloid into $1\times R^{-1/2}\times R^{-1}$ curved large-caps. In this article, we prove the sharp $\ell^2(L^{10/3})$ decoupling estimate for strips in $\R^3$. We reduce the strip decoupling inequalities to a planar incidence problem, which is dealt with via two-ends Furstenberg incidence estimates. Our method is inspired by Wang--Wu\cite{wang2024restrictionestimatesusingdecoupling} on the restriction conjecture.
\end{abstract}

\maketitle

\section{Introduction}\label{sec:Intro}

For a box $\Omega\subset \R^3$ and any $F:\R^3\to \C$, let $\psi_\Omega\in C^\infty(\R^3)$ be a bump function adapted to $\Omega$. Define $\wh{F_\Omega} = \wh{F}\psi_{\Omega}$. Consider the paraboloid in $\R^3$ $$\mathcal{M} = \{(\xi_1, \xi_2, \xi_1^2 +\xi_2^2)\; :\; |\xi_i|\le 1\}$$ and let $\mathcal{M}(R^{-1})$ be the $R^{-1}$-neighborhood of $\M$. Let $\{\theta\}$ be the partition of $\M(R^{-1})$ into the canonical caps of size $R^{-1/2}\times R^{-1/2}\times R^{-1}$ and let $\{S\}$ be a partition of $\M(R^{-1})$ into the curved large-caps of size $1\times R^{-1/2}\times R^{-1}$. That is, each $S$ is the intersection of $\M(R^{-1})$ with a box of dimensions $1\times R^{-1/2}\times 1$. See Figure \ref{fig:strips}. Moreover, each $S$ is essentially a union of canonical caps $\theta$ whose centers are coplanar. 

Our main result is the $\ell^2(L^{10/3})$ decoupling inequality into strips $\{S\}$. This simultaneously answers Question 3.4 ($\ell^2(L^3)$ strip decoupling) and Question 3.5 ($\ell^{10/3}(L^{10/3})$ strip decoupling) of Demeter in \cite{demeter2026BeyondCanonical}.

\begin{theorem}\label{thm:stripDec}
    For each $B_R\subset \R^3$, $F: \R^3\to \C$ with $\supp \wh{F}\subset \mathcal{M}(R^{-1})$ and each $2\le p\le \frac{10}{3}$, we have $$\|F\|_{L^p(B_R)}\lessapprox \left(\sum_{S} \|F_S\|_{L^p(w_{B_R})}^2\right)^{1/2},$$ where $F_S= \sum_{\theta \subset S}F_\theta$.
\end{theorem}

A bush example shows that $\frac{10}{3}$ is optimal here; see subsection \ref{subsect:Examples}. By a rotational transformation, we may assume the strips are narrow in the $\xi_2$ direction. Let $\{I\}$ be a partition of $[-1,1]$ into intervals of length $R^{-1/2}$. Then, each strip $S$ is the $R^{-1}$-neighborhood of $$\{(\xi_1, \xi_2, \xi_1^2+\xi_2^2) \; : |\xi_1|\le 1, \xi_2\in I\}$$ for some $R^{-1/2}$-interval $I$; see Figure \ref{fig:strips}.

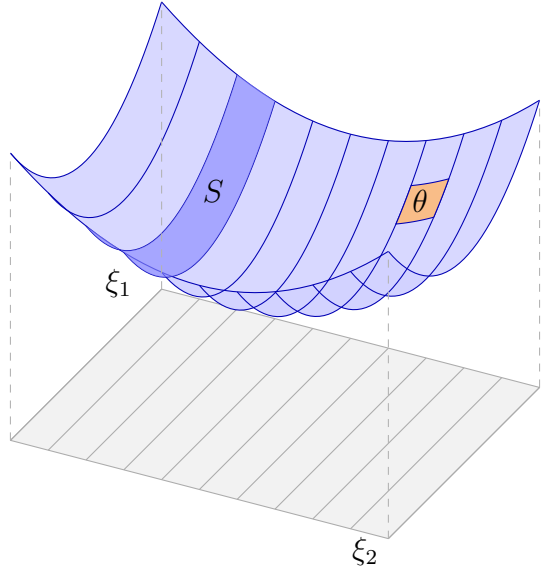
\begin{figure}
    \centering
    \begin{tikzpicture}[
  x={(2.5cm,-0.65cm)},
  y={(1.0cm,1.0cm)},
  z={(0cm,2.0cm)},
  line join=round
]
  \filldraw[fill=gray!10,draw=gray!65]
    (-1,-1,0) -- (1,-1,0) -- (1,1,0) -- (-1,1,0) -- cycle;

  \foreach \a in {-1,-0.8,...,1}{
    \draw[gray!65,thin] (\a,-1,0) -- (\a,1,0);
  }

  \foreach \a/\b in {
    -1/-0.8, -0.8/-0.6, -0.6/-0.4, -0.4/-0.2, -0.2/0,
    0/0.2, 0.2/0.4, 0.4/0.6, 0.6/0.8, 0.8/1
  }{
    \fill[blue!18,opacity=0.85]
      plot[domain=-1:1,samples=25,variable=\t]
        ({\a},{\t},{0.8+0.55*(\a*\a+\t*\t)})
      --
      plot[domain=1:-1,samples=25,variable=\t]
        ({\b},{\t},{0.8+0.55*(\b*\b+\t*\t)})
      -- cycle;
  }

  \foreach \a in {-1,-0.8,...,1}{
    \draw[blue!70!black,thin]
      plot[domain=-1:1,samples=35,variable=\t]
        ({\a},{\t},{0.8+0.55*(\a*\a+\t*\t)});
  }

  \foreach \b in {-1,1}{
    \draw[blue!70!black]
      plot[domain=-1:1,samples=35,variable=\t]
        ({\t},{\b},{0.8+0.55*(\t*\t+\b*\b)});
  }

  \fill[blue!45,opacity=0.65]
    plot[domain=-1:1,samples=35,variable=\t]
      ({-0.6},{\t},{0.8+0.55*(\t*\t+0.6*0.6)})
    --
    plot[domain=1:-1,samples=35,variable=\t]
      ({-0.4},{\t},{0.8+0.55*(\t*\t+0.4*0.4)})
    -- cycle;

  \filldraw[fill=orange!55,fill opacity=0.85,
            draw=blue!70!black,thin]
    plot[domain=0.4:0.6,samples=12,variable=\t]
      ({\t},{0.6},{0.8+0.55*(\t*\t+0.6*0.6)})
    --
    plot[domain=0.6:0.8,samples=12,variable=\t]
      ({0.6},{\t},{0.8+0.55*(0.6*0.6+\t*\t)})
    --
    plot[domain=0.6:0.4,samples=12,variable=\t]
      ({\t},{0.8},{0.8+0.55*(\t*\t+0.8*0.8)})
    --
    plot[domain=0.8:0.6,samples=12,variable=\t]
      ({0.4},{\t},{0.8+0.55*(0.4*0.4+\t*\t)})
    -- cycle;

  \foreach \a in {-1,1}{
    \foreach \b in {-1,1}{
      \draw[gray!60,dashed]
        (\a,\b,0) -- (\a,\b,{0.8+0.55*(\a*\a+\b*\b)});
    }
  }
  \node[above left] at (-1,0.75,0) {$\xi_1$};
  \node[below right]  at (0.75,-1,0) {$\xi_2$};

  \node[below right]  at (-0.74,0.75,1) {$S$};
  \node[above right]  at (0.43,0.6,1.1) {$\theta$};
\end{tikzpicture}
    \caption{Strips on the paraboloid in $\R^3$}
    \label{fig:strips}
\end{figure}

Strips have presented a major obstacle due to their anisotropic scaling; that is, parabolic rescaling changes the dimensions of curved large-caps. Moreover, pigeonholing or pruning strips may destroy some of the curvature in the support of $\wh{F}$ which is crucial for the desired decoupling. We wish to work with canonical caps $\theta$ to avoid these challenges. From this, we are motivated to study the visually stronger, but equivalent, statement of Theorem \ref{thm:strongStripDec}. This stands as an intermediate between canonical decoupling and reverse square-function inequalities. For convenience, define $$\bS_SF = \left(\sum_{\theta\subset S}|F_\theta|^2\right)^{1/2}$$ for each strip $S$.
\begin{theorem}\label{thm:strongStripDec}
    For each $B_R\subset \R^3$, $F: \R^3\to \C$ with $\supp \wh{F}\subset \mathcal{M}(R^{-1})$ and each $2\le p\le \frac{10}{3}$, we have $$\|F\|_{L^p(B_R)}\lessapprox \left(\sum_{S} \left\|\bS_SF\right\|_{L^p(w_{B_R})}^2\right)^{1/2}.$$
\end{theorem}
By Rubio de Francia square-function estimates (for instance, see \cite[Theorem 5.16]{demeter2020fourier}) and the reverse square-function estimate for the parabola (see \cite[Prop. 3.3]{demeter2020fourier}) with Fubini's Theorem (i.e., cylindrical decoupling), we have for $2\le p\le 4$, $$\|\bS_S F\|_{L^p(w_{B_R})}\sim \|F_S\|_{L^p(w_{B_R})}.$$ Hence, Theorem \ref{thm:stripDec} is equivalent to Theorem \ref{thm:strongStripDec}. 

Demeter proved an $\ell^p(L^p)$ variant of Theorem \ref{thm:stripDec} for the range $2\le p\le 3$ \cite[Theorem 3.1]{demeter2026BeyondCanonical} using the trilinear $L^3$ restriction estimate. Indeed, the trilinear $L^3$ restriction estimate shows a trilinear $\ell^2(L^3)$ version of Theorem \ref{thm:strongStripDec} as seen in \cite[Theorem 3.3]{demeter2026BeyondCanonical}. Unfortunately, the broad-narrow argument destroys the $\ell^2$ structure of the narrow term. We must overcome this obstacle in our proof of Theorem \ref{thm:strongStripDec} by avoiding a broad-narrow argument altogether.

These results serve as a special case of the broader program which began in \cite{demeter2026BeyondCanonical} to study the restriction conjecture: for $p\ge 3$, \begin{equation}\label{eq:L^p restriction}
    \|Ef\|_{L^p(B_R)}\lessapprox \|f\|_{\infty},
\end{equation} where $E$ is the extension operator for the paraboloid $\M$ in $\R^3$. For completeness, let us state the $\ell^p$ large-ensemble decoupling conjecture which implies the restriction estimate \eqref{eq:L^p restriction}. We call a union of $M$ canonical caps an \textit{$M$-set}. We call a collection of $M$-sets $\{P\}$ an \textit{$M$-large-ensemble collection} if $\{P\}$ forms a partition of $\M(R^{-1})$. 

\begin{conj}\label{conj:MlargeEnsemble}
    Fix $p\in (3,4]$ and $M\sim R^{\frac{4-p}{p-2}}$. For each $B_R\subset \R^3$, $F: \R^3\to \C$ with $\supp \wh{F}\subset \mathcal{M}(R^{-1})$, each $M$-large-ensemble collection $\{P\}$, we have $$\|F\|_{L^p(B_R)}\lessapprox (R/M)^{\frac{1}{2}-\frac{1}{p}}\left(\sum_{P} \|F_P\|_{L^p(w_{B_R})}^p\right)^{1/p}.$$
\end{conj}
To gauge the expected difficulty of Conjecture \ref{conj:MlargeEnsemble}, note $p=\frac{10}{3}$ and $p=\frac{22}{7}$ correspond to $M\sim R^{1/2}$ and $M\sim R^{3/4}$, respectively. These would imply the restriction estimates of Tao \cite{Tao03} and Wang--Wu \cite{wang2024restrictionestimatesusingdecoupling}. The full restriction problem would follow from Conjecture \ref{conj:MlargeEnsemble} with $M\sim R^{1-o(1)}$.

Let us briefly overview the proof of Theorem \ref{thm:strongStripDec}. Using the square-functions $\bS_S F$, we may express each $$\|\bS_SF\|_{L^p(w_{B_R})}$$ in terms of the incidences of the tubes coming from the canonical wave packet decomposition of $F$. Since the tubes coming from each $\theta\subset S$ can be partitioned in families of coplanar tubes, we may use planar incidence geometry to analyze them. The technical heart of our proof is inspired by Wang--Wu \cite{wang2024restrictionestimatesusingdecoupling}. We follow a two-ends induction-on-scales approach. Using the canonical wave packet decomposition of a function $F$, we first refine the wave packet statistics to make them uniform in various ways and establish "two-endedness" of the tubes with some shading of cubes. The uniformity allows us to gain from the $L^4$ refined decoupling inequality. We interpolate between $L^2$ orthogonality and $L^4$ refined decoupling. This reduces the strip decoupling to a stripwise refined multiplicity estimate Theorem \ref{thm:multMoment}. As mentioned above, the crucial observation here is that the tubes from a strip can be partitioned into coplanar families. Therefore, we may use planar incidence geometry, namely the two-ends Furstenberg estimates.

The outline of this paper is as follows. For the rest of Section \ref{sec:Intro}, we state a connection with small-cap decoupling inequalities and give the bush example to show $\frac{10}{3}$ is optimal for Theorem \ref{thm:stripDec}. In Section \ref{sec:p<3}, we give the author's original arguments for Theorem \ref{thm:strongStripDec} in the range $2\le p\le 3$. This will serve to highlight the key steps of our approach. Sections \ref{sec:p>3} and \ref{sec:StripTubeGeo} will give the proof of Theorem \ref{thm:strongStripDec} in the remaining range $3<p\le \frac{10}{3}$. In particular, Section \ref{sec:p>3} will reduce the strip decoupling inequality to an incidence problem which will be established in Section \ref{sec:StripTubeGeo}.

\subsection{Application to Small-Cap Decoupling}
As an application of our result, we can draw a connection to the recent small-cap decoupling inequalities of Guth--Maldague--Oh \cite{GMO25}. Let's recall a special instance of their result. Let $\{\gamma\}$ be a partition of $\M(R^{-1})$ into small-caps of size $R^{-1/2}\times R^{-1}\times R^{-1}$.
\begin{theorem}\label{thm:GMO}
    \cite{GMO25} For each $F: \R^3\to \C$ with $\supp \wh{F}\subset \mathcal{M}(R^{-1})$ and each $2\le p\le \frac{10}{3}$, we have $$\|F\|_{L^p(\R^3)}\lessapprox (R^{3/2})^{\frac{1}{2}-\frac{1}{p}}\left(\sum_{\gamma} \|F_\gamma\|_{L^p(\R^3)}^p\right)^{1/p}.$$
\end{theorem}

Surprisingly, we can give an alternate proof of Theorem \ref{thm:GMO}; we actually prove something slightly stronger. This suggests a correspondence between large-cap and small-cap decoupling inequalities. 

\begin{corollary}
    For each $B_R\subset \R^3$, $F: \R^3\to \C$ with $\supp \wh{F}\subset \mathcal{M}(R^{-1})$ and each $2\le p\le \frac{10}{3}$, we have $$\|F\|_{L^p(B_R)}\lessapprox R^{\frac{1}{2}-\frac{1}{p}} \left(\sum_{S} \left(\sum_{\gamma\subset S}\|F_\gamma\|_{L^p(w_{B_R})}^p\right)^{\frac{2}{p}}\right)^{1/2}.$$
\end{corollary}

 \begin{proof}
     Apply Theorem \ref{thm:stripDec}. Then, since $\frac{10}{3}\le 4$, we can apply the $\ell^p(L^p)$ small-cap decoupling for the parabola (see \cite[Theorem 2.3]{demeter2020smallcapdecouplings}) on each strip.
 \end{proof}

 Further, as noted in \cite{GMO25}, Theorem \ref{thm:GMO} implies the restriction estimate \eqref{eq:L^p restriction} for $p= \frac{10}{3}$. Hence, Theorem \ref{thm:stripDec} is alone strong enough for the restriction estimate at $p=\frac{10}{3}$, without needing the full strength of Conjecture \ref{conj:MlargeEnsemble} with $M\sim R^{1/2}$.

\subsection{Sharp Example for Theorem \ref{thm:stripDec}}\label{subsect:Examples}

Demeter used this bush example in \cite{demeter2026BeyondCanonical} to show $\frac{10}{3}$ is optimal for the $\ell^p(L^p)$ strip decoupling. For the sake of the reader, we give this example here and show it works as an example for Theorem \ref{thm:stripDec}. Let $$F=\sum_\theta F_\theta$$ with each $F_\theta$ a single wave packet containing the origin, with magnitude $\sim 1$. So, $|F_\theta| \approx 1_{T_\theta}$. By constructive interference, $$\|F\|_{L^p(B(0,R))}\gtrsim R.$$ Alternatively, by the bush estimate for tubes (see \cite[Exercise 5.9]{demeter2020fourier}), $$\|\bS_SF\|_{L^p(B(0,R))}\lessapprox \|\sum_{\theta \subset S}1_{T_\theta}\|_{p/2}^{1/2}\sim R^{\frac{5}{2p}}.$$ Thus, Theorem \ref{thm:stripDec} can only hold if $$R^p\lessapprox (R^{1/2})^{\frac{p}{2}}R^{\frac{5}{2}} \iff p\le \frac{10}{3}.$$ 

\subsection{Notation}
Throughout this paper, we will use the following notation for inequalities with implicit constants. We write $A\lesssim B$ if there is some $c>0$, independent of $R$, so that $A\le c B$ and we write $A\sim B$ if $A\lesssim B$ and $B\lesssim A$. Further, we write $A\lessapprox B$ if for arbitrarily small $\eta>0$, there is some constant $C_\eta>0$ so that $$A\le C_\eta R^{\eta} B.$$ For instance, $A\lesssim (\log R)^{O(1)}$ can be written as $A\lessapprox 1$. We also let $w_{B_R}$ be a smooth weight such that $w_{B_R}\sim 1$ on $B_R$ and rapidly decays away from $B_R$.

\medskip
\noindent
{\it Acknowledgments}. The author would like to thank his advisor, Ciprian Demeter, for introducing him to this problem and for discussions throughout the completion of this work. The author would also like to thank Shukun Wu for some insightful and enriching remarks early in the development of this work.

\medskip
\noindent
{\it Use of AI declaration}. During the development of this work, artificial intelligence had a varied role. Initially, the author proved Theorem \ref{thm:stripDec} in the range $2\le p\le 3$ and the $\ell^p(L^p)$ strip decoupling inequality for the full range $2\le p\le \frac{10}{3}$ without the aid of artificial intelligence. ChatGPT 5.6 Sol was asked to extend these methods to prove Theorem \ref{thm:stripDec} in the range $3<p\le \frac{10}{3}$, naturally using the stronger two-ends incidence estimates. While having technical gaps and many obtuse arguments, it provided the key ideas for Subsection \ref{subsec:NonTwoEnds} and Section \ref{sec:StripTubeGeo}. This note was solely written and edited by the author. ChatGPT 5.6 and 6 Sol were also used to create the Figures throughout; these were originally hand-drawn and ChatGPT was asked to create the TikZ code from the drawing. All mathematical claims and verification remain the work and responsibility of the author.

\section{Proof of Theorem \ref{thm:strongStripDec} for $p\le 3$}\label{sec:p<3}
First, to explain the ideas, we will give a proof of our main theorem for $2\le p\le 3$. Fix $B_R\subset \R^3$, $F$ with $\supp \wh{F}\subset \mathcal{M}(R^{-1})$, and $2\le p\le 3$.

\subsection{Wave Packet Decomposition and Pigeonholing}\label{subsect:wavepackets3}

Recall the wave packet decomposition of $F$ (see \cite{demeter2020fourier}): for each canonical cap $\theta$, $$F_\theta = \sum_{T\in \T_\theta} F_T,$$ where each $T$ is a tube of dimensions $R^{1/2}\times R^{1/2}\times R$ dual to $\theta$ and $|F_T|$ is essentially constant and essentially supported on $T$. Let $\T = \bigcup_\theta \T_\theta$ and for each strip $S$, let $\T_S= \bigcup_{\theta\subset S}\T_\theta$. By standard pigeonholing, we may assume $\|F_T\|_\infty \sim a$ for each $T\in \T$. 

Moreover, by the following lemma, we may assume $F$ has other uniform "statistics." We are aiming to reduce $B_R$ to a collection of $R^{1/2}$-cubes $\sQ=\{Q\}$ and the tubes to a subcollection $\T^{'}$ where both $$\#\T^{'}(Q)=\#\{T\in \T^{'}: Q\subset T\}\sim m, $$ and $$ \#\sQ(T)=\#\{Q\in \sQ : Q\subset T \}\sim \lambda R^{1/2}.$$ This would allow us to apply $L^4$ refined decoupling and gain by a power of $\lambda$ with $L^2$. We cannot literally reduce to this configuration, but we can reduce to a sufficiently close configuration. When applying this lemma, we may take $\eta$ sufficiently small so that we only get $\lessapprox 1$ losses.

\begin{lemma}\label{label:uniformStats}
   Let $\eta>0$. There exist $X\subset B_R$, a union of $R^{1/2}$-cubes $\sQ_X$ and a subset of tubes $\T^{'}\subset \T$ such that $$\|F\|_{L^p(B_R)}\lessapprox \|\tilde{F}\|_{L^p(X)}$$ where $\tilde{F} = \sum_{T\in \T^{'}}F_T$ and such that \begin{enumerate}
        \item for each $T\in \T^{'}$, $\#\sQ_X(T)\sim \lambda R^{1/2}$ for some dyadic $R^{-1/2}\lesssim \lambda \lesssim 1$,

        \item for each $Q\in \sQ_X$, $\#\T^{'}(Q)\lesssim m$,

        \item $R^{-\eta}\#\sQ_X m\lesssim\lambda R^{1/2}\#\T^{'}\lesssim \#\sQ_X m$
    \end{enumerate}
\end{lemma}

\begin{proof}
    We will follow an iterative process. We will pigeonhole to create (1) and (2). If (3) is satisfied, we may stop. If (3) is not satisfied, then we have lost a significant portion of the cube-tube incidences and we can repeat. We must stop after only $O(\eta^{-1})$ iterations since we would run out of cube-tube incidences. 
    
    \textbf{Initial step}: let $X_0=B_R$, $\T_0=\T$, $\lambda_0= 1$, $m_0=R$ and $F_0= \sum_{T\in \T_0}F_T$. Note $$ \lambda_0R^{1/2}\#\T_0 \lesssim \#\sQ_{X_0}m_0.$$ If also we have\begin{equation}\label{eqn:stoppingCondition}
         \#\sQ_{X_0}m_0\lesssim R^{\eta}\lambda_0R^{1/2}\#\T_0,
    \end{equation} we are done by taking $X=X_0$, $\T^{'}=\T_0$, $\lambda=\lambda_0$, and $m=m_0$. If not, we can repeat the pigeonholing process.
    
    \textbf{Inductive step}: we will now construct $(X_{j+1}, \T_{j+1})$ given $(X_j, \T_j)$.
    Pigeonhole the cubes to create $X_{j+1}\subset X_j$ such that for each $Q\subset X_{j+1}$, we have $\#\T_{j}(Q)\sim m_{j+1}$ for some dyadic $m_{j+1}$ and $$\|F_{j}\|_{L^p(X_{j})}^p\lesssim (\log R)\|F_{j}\|_{L^p(X_{j+1})}^p.$$ Now, pigeonhole the tubes to create $\T_{j+1}\subset \T_j$ such that for each $T\in \T_{j+1}$, we have $\#\sQ_{X_{j+1}}(T) \sim \lambda_{j+1}R^{1/2}$ for some dyadic $\lambda_{j+1}$ and $$\|F_j\|_{L^p(X_{j+1})}\lesssim (\log R)\|F_{j+1}\|_{L^p(X_{j+1})}.$$ As in the initial step, we have $$\lambda_{j+1}R^{1/2}\#\T_{j+1}\sim I(\sQ_{X_{j+1}}, \T_{j+1})\le I(\sQ_{X_{j+1}}, \T_{j})\sim \#\sQ_{X_{j+1}}m_{j+1}.$$ If also we have \begin{equation}\label{eqn:stoppingConditionGenereal}
          \#\sQ_{X_{j+1}}m_{j+1}\lesssim R^{\eta}\lambda_{j+1}R^{1/2}\#\T_{j+1},
    \end{equation} we may stop with $X=X_{j+1}$, $\T^{'}=\T_{j+1}$, $\lambda=\lambda_{j+1}$, and $m=m_{j+1}$. 
    
    \textbf{Termination}: every iteration which does not terminate the process satisfies $$\lambda_{j+1}R^{1/2}\#\T_{j+1} \lesssim R^{-\eta} \#\sQ_{X_{j+1}}m_{j+1} \sim I(\sQ_{X_{j+1}}, \T_{j})\le I(\sQ_{X_j}, \T_j) \sim \lambda_j R^{1/2}\#\T_j.$$ Therefore, if we iterate $K$ times, $$1\lesssim \lambda_KR^{1/2}\#\T_K\lesssim R^{-K\eta}\lambda_0R^{1/2}\#\T_0\lesssim R^{O(1)-K\eta}.$$ Hence, we need only iterate at most $K\lesssim \eta^{-1}$ times before terminating. Since each iteration gives a factor of $(\log R)^{1+\frac{1}{p}}\lesssim (\log R)^2 $, we get $$\|F\|_{L^p(B_R)}\lesssim (\log R)^{O(\eta^{-1})}\|\tilde{F}\|_{L^p(X)}.$$ We may use the fact that $\log R\lesssim_\epsilon R^{\epsilon}$ and take $\epsilon$ sufficiently small compared to $\eta$.
\end{proof}

We will abuse our notation and write $\T$ and $F$ for $\T^{'}$ and $\tilde{F}$, respectively, for the rest of this section. For each $Q\subset B_R$ and each strip $S$, let $$m(Q) = \#\{T\in \T : Q \text{ intersects } T\},$$ and  $$m_S(Q)=\#\{T \in \T_S : Q \text{ intersects } T\}.$$ We wanted these uniform statistics so we can apply the refined decoupling theorem proved by Guth--Iosevich--Ou--Wang in \cite{GIOWrefinedDecoupling}.

\begin{theorem}\label{thm:refinedDecoupling}
    Suppose $F$ is a sum of wave packets $$F=\sum_{T\in \T}F_T$$ and $X\subset B_R$ is a union of $R^{1/2}$-cubes such that for each $Q\subset Y$, $m(Q)\lesssim m$. Then, $$\|F\|_{L^4(Y)}\lessapprox m^{\frac{1}{4}}\left(\sum_T \|F_T\|_{L^4(B_R)}^4\right)^{1/4}.$$
\end{theorem}
\subsection{$L^p$ estimates}
By H\"{o}lder's inequality, we will interpolate between $L^2$ and $L^4$. We will then apply Minkowski's inequality to isolate each strip. 
\begin{prop}\label{prop:LpLHS3}
$$\|F\|_{L^p(X)}\lessapprox aR^{\frac{3}{2p}}\lambda^{-\frac{1}{2}+\frac{1}{p}}\left(\sum_{S}(\sum_{Q\in \sQ_X}m_S(Q)^{p/2})^{2/p}\right)^{1/2}.$$    
\end{prop}

\begin{proof}
    From Lemma \ref{label:uniformStats}, $$\lambda R^{1/2}\#\T\approx m\#\sQ_X.$$ By H\"{o}lder's inequality, $$\|F\|_{L^p(X)}\le \|F\|_{L^2(X)}^{\frac{4}{p}-1}\|F\|_{L^4(X)}^{2-\frac{4}{p}}.$$
By $L^2$ orthogonality and $L^4$ refined decoupling Theorem \ref{thm:refinedDecoupling}, $$\|F\|_{L^2(X)}^2 = \sum_{Q\in \sQ_X}\|F\|_{L^2(Q)}^2\lesssim \sum_{Q\in \sQ_X}\sum_{T}\|F_T\|_{L^2(Q)}^2\lesssim \#\sQ_X m a^2R^{3/2},$$
and $$\|F\|_{L^4(X)}^4 \lessapprox m\sum_{T}\|F_T\|_{L^4(B_R)}^4\sim m \#\T a^4R^{2}\lesssim\lambda^{-1}m^2\#\sQ_X a^4R^{3/2}.$$ Therefore, we get \begin{align*}
    \|F\|_{L^p(X)}&\lessapprox (\#\sQ_X m a^2R^{3/2})^{\frac{1}{2}(\frac{4}{p}-1)}(\lambda^{-1}m^2\#\sQ_X a^4R^{3/2})^{\frac{1}{4}(2-\frac{4}{p})}\\
    &= \lambda^{-\frac{1}{2}+\frac{1}{p}}m^{\frac{1}{2}}(\#\sQ_X)^{\frac{1}{p}}aR^{\frac{3}{2p}}.
\end{align*}
By H\"{o}lder's and Minkowski's inequalities, \begin{multline*}
    m^{1/2}(\#\sQ_X)^{1/p} \lessapprox (\lambda R^{1/2}\#\T)^{1/2} (\#\sQ_X)^{\frac{1}{p}-\frac{1}{2}}\sim \frac{1}{(\#\sQ_X)^{\frac{1}{2}-\frac{1}{p}}}\left(\sum_{Q\in \sQ_X}m(Q)\right)^{1/2}\\
    \lesssim \left(\sum_{Q\in \sQ_X}m(Q)^{p/2}\right)^{1/p}
    \le \left(\sum_{S}(\sum_{Q\in \sQ_X}m_S(Q)^{p/2})^{2/p}\right)^{1/2}.
\end{multline*}
\end{proof}
 
\begin{lemma}\label{lem:LpRHS} For each strip $S$,
    $$\left\|\bS_SF\right\|_{L^p(B_R)} \gtrsim R^{\frac{3}{2p}}a\left(\sum_{Q\subset B_R}m_S(Q)^{\frac{p}{2}}\right)^{\frac{1}{p}}.$$
\end{lemma}
\begin{proof}
    Note that $$\left\|\bS_SF\right\|_{L^p(B_R)}^p\gtrsim \sum_{Q\subset B_R}\left\|\left(\sum_{T\in \T_S}|F_T|^2\right)^{1/2}\right\|_{L^p(Q)}^p\sim\sum_{Q\subset B_R}|Q|a^p m_S(Q)^{\frac{p}{2}}.$$
\end{proof}

\begin{corollary}\label{cor:MultMomentReduction3}
    Suppose for each strip $S$,
    \begin{equation}\label{eqn:StripwiseIncidenceMoment}
        \sum_{Q\in \sQ_X}m_S(Q)^{p/2}\lessapprox \lambda^{\frac{p}{2}-1}\sum_{Q\subset B_R}m_S(Q)^{p/2}.
    \end{equation}
    Then, Theorem \ref{thm:strongStripDec} holds.
\end{corollary}

\begin{proof}
    Observe that by Lemma \ref{lem:LpRHS} we have \begin{equation}\label{eqn:RHS}
        \left(\sum_{S} \left\|\bS_SF\right\|_{L^p(B_R)}^2\right)^{1/2}\gtrsim R^{\frac{3}{2p}}a\left(\sum_S \left(\sum_{Q\subset B_R}m_S(Q)^{p/2}\right)^{2/p}\right)^{1/2}.
    \end{equation} Now, apply Proposition \ref{prop:LpLHS3} and the strip assumption \eqref{eqn:StripwiseIncidenceMoment} to see that this gives Theorem \ref{thm:strongStripDec}.
\end{proof}

Therefore, to prove Theorem \ref{thm:strongStripDec}, it suffices to show \eqref{eqn:StripwiseIncidenceMoment} given the tube-cube statistics from Subsection \ref{subsect:wavepackets3}, namely $\#\sQ_X(T)\sim \lambda R^{1/2}$ for each $T\in \T$.  

\subsection{Kakeya-strength Incidence Geometry}

The estimate \eqref{eqn:StripwiseIncidenceMoment} follows from the following incidence result.
\begin{theorem}\label{thm:multMomentKakeya}
    Let $\T$ be a collection of tubes in $\R^3$ from a strip $S$ and $X$ the union of $R^{1/2}$-cubes in $B_R$ from a collection $\sQ$ such that $\#\sQ(T)\sim \lambda R^{1/2}$ for each $T\in \T$. Then, for each $2< p\le 4$, $$\sum_{Q\in \sQ}m(Q)^{p/2}\lessapprox \lambda^{2-\frac{p}{2}}\sum_{Q\subset B_R}m(Q)^{p/2}.$$
\end{theorem}

\begin{proof}[Proof of Theorem \ref{thm:strongStripDec} for $p\le 3$]
    We apply Theorem \ref{thm:multMomentKakeya} to Corollary \ref{cor:MultMomentReduction3} and note that $$\lambda^{2-\frac{p}{2}}\lesssim \lambda^{\frac{p}{2}-1} \iff p\le 3.$$
\end{proof}

Now, we dedicate the rest of this section to proving Theorem \ref{thm:multMomentKakeya}. Fix a strip $S$. Note that the tubes in $\T$ have directions which are essentially planar. Indeed, suppose $\theta\subset S$ is centered at $(\xi_1,c_S, \xi_1^2+c_S^2)$. The long direction of the tubes in $\T_\theta$ is $(2\xi_1, 2c_S, -1)$. Hence, since we may assume the strips are narrow in the $\xi_2$-direction, these directions for $\theta \subset S$ are captured by the plane normal to $(0, 1, 2c_S)$.

\begin{figure}
    \centering
    \begin{tikzpicture}[
   x={(0.866cm,-0.500cm)},
  y={(0.52cm,0.32cm)},
  z={(0.500cm,0.866cm)},
  line cap=round,
  line join=round,
  scale=0.7
]

\fill[cyan!18,opacity=.38]
  (0,0,5)--(6,0,5)--(6,0,6)--(0,0,6)--cycle;
\fill[cyan!12,opacity=.32]
  (6,0,5)--(6,6,5)--(6,6,6)--(6,0,6)--cycle;
\fill[cyan!8,opacity=.24]
  (0,6,5)--(6,6,5)--(6,6,6)--(0,6,6)--cycle;

\foreach \h in {1,2,3,4,5}{
  \draw[black!55,line width=.45pt]
    (0,0,\h)--(6,0,\h)--(6,6,\h)--(0,6,\h)--cycle;
}

\newcommand{\tube}[2]{%
  \draw[blue!65!black,line width=3pt] #1 -- #2;
  \draw[cyan!20,line width=.65pt,opacity=.9] #1 -- #2;
  \filldraw[
    fill=cyan!12,
    draw=blue!65!black,
    line width=.45pt
  ] #1 circle[radius=1.55pt];
  \filldraw[
    fill=cyan!12,
    draw=blue!65!black,
    line width=.45pt
  ] #2 circle[radius=1.55pt];
}

\tube{(0.35,0.75,5.22)}{(5.65,5.20,5.72)}
\tube{(0.55,5.45,5.30)}{(5.55,0.55,5.78)}
\tube{(0.85,0.35,5.76)}{(5.45,5.55,5.18)}
\tube{(0.35,3.85,5.61)}{(5.70,2.10,5.37)}
\tube{(2.15,0.30,5.16)}{(3.05,5.72,5.84)}
\tube{(4.75,0.40,5.80)}{(1.65,5.55,5.27)}

\fill[cyan!8,opacity=.12]
  (0,0,6)--(6,0,6)--(6,6,6)--(0,6,6)--cycle;

\draw[black,line width=.85pt]
  (0,0,0)--(6,0,0)--(6,6,0)--(0,6,0)--cycle;
\draw[black,line width=.85pt]
  (0,0,6)--(6,0,6)--(6,6,6)--(0,6,6)--cycle;

\foreach \x/\y in {0/0,6/0,6/6,0/6}{
  \draw[black,line width=.85pt]
    (\x,\y,0)--(\x,\y,6);
}

\draw[blue!45!black,line width=.9pt]
  (0,0,5)--(6,0,5)--(6,6,5)--(0,6,5)--cycle;

\end{tikzpicture}
    \caption{The plates $U$ and the tubes in $\T=\T_S$ captured by a plate.}
    \label{fig:plates}
\end{figure}
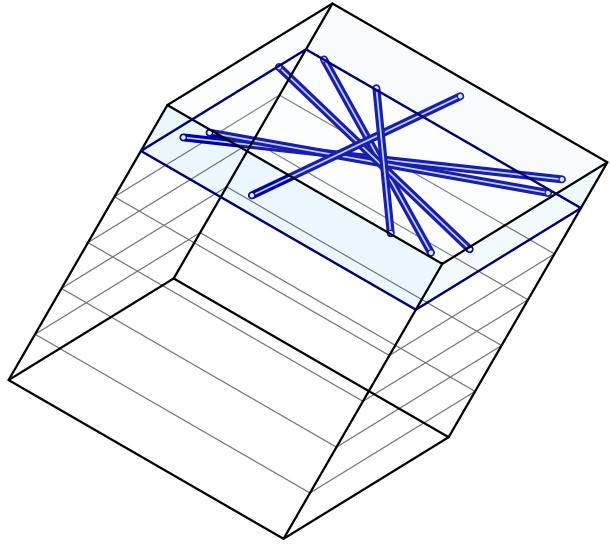

Tile $B_R$ with $R\times R\times R^{1/2}$ plates $U$ that are parallel to this plane; see Figure \ref{fig:plates}. Note that a tube $T\in \T$ is contained in essentially one plate $U$ and tubes in distinct plates are essentially disjoint. Fix one of these plates $U$. It suffices to show $$\sum_{Q\subset X\cap U}m(Q)^{p/2}\lessapprox \lambda^{2-\frac{p}{2}}\sum_{Q\subset U}m(Q)^{p/2}.$$ We will now restrict $\T$ and $X$ to one of these plates $U$. Therefore, we are left with the following planar incidence problem.

\begin{theorem}\label{thm:planarIncidenceKakeya}
    Let $\T$ be a collection of tubes in $\R^2$ and $X$ the union of $R^{1/2}$-squares in $B_R$ from a collection $\sQ$ such that $\#\sQ(T)\sim \lambda R^{1/2}$ for each $T\in \T$. Then, for each $2< p\le 4$, $$\sum_{Q\in \sQ}m(Q)^{p/2}\lessapprox \lambda^{2-\frac{p}{2}}\sum_{Q\subset B_R}m(Q)^{p/2}.$$
\end{theorem}

We now apply a planar Kakeya argument to prove Theorem \ref{thm:planarIncidenceKakeya} and hence, finish the proof of Theorem \ref{thm:multMomentKakeya}.
\begin{proof}

First, an application of H\"{o}lder's inequality:
\begin{align*}
    \sum_{Q\in \sQ}m(Q)^{p/2} &= \sum_{Q\in \sQ}m(Q)^{\frac{p}{2}-1}\sum_{T: Q\cap T\neq \emptyset}1\\
    &= \sum_{T} \sum_{Q\in \sQ(T)}m(Q)^{\frac{p}{2}-1}\\
    &\le \sum_T (\sum_{Q\in \sQ(T)}m(Q))^{\frac{p}{2}-1}(\sum_{Q\in \sQ(T)}1)^{2-\frac{p}{2}}\\
    &\lesssim \lambda^{2-\frac{p}{2}}(R^{1/2})^{2-\frac{p}{2}}\sum_T (\sum_{Q: Q\cap T\neq \emptyset}m(Q))^{\frac{p}{2}-1}.
\end{align*}
Define for each $T\in \T$, the hairbrush quantity $$H_T = \sum_{Q: Q\cap T\neq \emptyset}m(Q)= \sum_{T^{'}}\sum_{Q\cap T\cap T^{'}\neq \emptyset}1= \sum_{T^{'}}c(T, T^{'}) \lesssim\sum_{\substack{1\le \rho \lesssim R^{1/2}\\ dyadic}}\rho \#\{T^{'}: c(T, T^{'})\sim \rho\}$$ where $c(T,T^{'})=\#\{Q : Q\cap T\cap T^\prime\neq \emptyset\}$. The above computation simplifies to $$\sum_{Q\in \sQ}m(Q)^{\frac{p}{2}}\lesssim \lambda^{2-\frac{p}{2}}(R^{1/2})^{2-\frac{p}{2}}\sum_T H_T^{\frac{p}{2}-1}.$$ Let $q=\frac{p}{2}-1\in (0,1]$. Note $$H_T^{q}\lessapprox \sum_{\substack{1\le \rho \lesssim R^{1/2}\\ dyadic}}\rho^{q}(\#\{T^{'}:c(T,T^{'})\sim \rho\})^{q}.$$ Note that  if $c(T, T^{'})\sim \rho$, then $T$ and $T^{'}$ lie in a common fat $R\times O(R\rho^{-1})$ tube. Let $\{P\}$ be a system of $R\times (R\rho^{-1})$ tubes covering $B_R$ (that is, there are $\sim \rho$ many fat directions and $\sim \rho$ many fat tubes $P$ per direction). Note that each $T\in \T$ is essentially contained in one fat tube $P$. Let $M(P) = \#\{T\in \T: T\subset P\}$. So, we have $$\sum_T(\#\{T^{'}:c(T,T^{'})\sim \rho\})^{q}\lesssim \sum_P \sum_{T\subset P}\left(\sum_{P^{'}\sim P}M(P^{'})\right)^{q}\lesssim \sum_{P, P^{'}: P\sim P^{'}}M(P)M(P^{'})^{q}$$ $$\lesssim \sum_{P}M(P)^{p/2},$$ where $P^{'}\sim P$ means they are neighboring fat tubes.
Now, by double counting the incidences inside of $P$, $$R^{1/2}M(P) \sim \sum_{Q\subset P}m_P(Q)$$ where $m_P(Q) = \#\{T\subset P : Q\cap T\neq \emptyset\}$. So, again by H\"{o}lder's inequality,\begin{align*}
    \sum_P M(P)^{p/2}&\sim (R^{1/2})^{-p/2}\sum_{P}(\sum_{Q\subset P} m_P(Q))^{p/2}\\
    &\lesssim(R^{1/2})^{-p/2}\sum_P \sum_Q m_P(Q)^{p/2}(R\rho^{-1})^{q}\\
    &= (R^{1/2})^{\frac{p}{2}-2}\rho^{-q}\sum_Q \sum_P m_P(Q)^{p/2}.
\end{align*} Since $p/2>1$, we have the trivial inequality $$\sum_P m_P(Q)^{p/2}\le (\sum_P m_P(Q))^{p/2}\lesssim m(Q)^{p/2}.$$ In summary, this gives $$\sum_{Q\in \sQ}m(Q)^{\frac{p}{2}}\lessapprox \lambda^{2-\frac{p}{2}}\sum_{Q\subset B_R} m(Q)^{p/2}.$$
\end{proof}

\begin{remark}
    Theorem \ref{thm:planarIncidenceKakeya} is essentially sharp. We give the example below. 
\end{remark}

Therefore, to extend the range to $3<p\le \frac{10}{3}$, we must deal with this enemy bush example; see Figure \ref{fig:BushExample}. This is done with the two-ends reduction which will give us the gain $\lambda^{\frac{3}{2}-\frac{p}{4}}$, instead of $\lambda^{2-\frac{p}{2}}$ present here. The enemy bush example will be captured using induction-on-scales in the non-two-ends case. Thus, the remaining two-ends case will give us an improved incidence result.

Let $\T$ be a bush about the origin coming from $N$ consecutive directions with $1\lesssim N\lesssim R^{1/2}$  and $\sQ$ be the collection of squares which are incident to all the tubes in $\T$. This gives $\lambda\sim \frac{1}{N}$. So, $$\sum_{Q\in \sQ}m(Q)^{p/2}\sim \frac{R^{1/2}}{N}(N)^{p/2}.$$ Similar to the calculations in Subsection \ref{subsect:Examples}, we have $$\sum_{Q\subset B_R}m(Q)^{p/2}\sim \frac{1}{R}\|\sum_{T\in \T}1_T\|_{p/2}^{p/2}\lessapprox NR^{1/2}.$$ Hence, we have $$\sum_{Q\in \sQ} m(Q)^{p/2}\approx \lambda^{2-\frac{p}{2}}\sum_{Q\subset B_R}m(Q)^{p/2}.$$
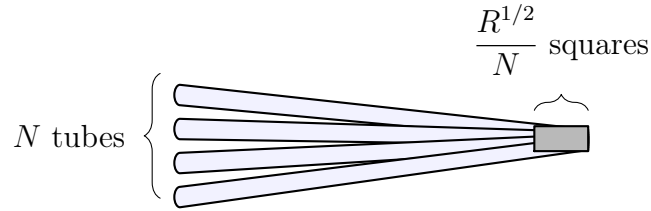
\begin{figure}[b!]
    \centering
    \begin{tikzpicture}[
    line cap=round,
    line join=round,
    tube/.style={draw=black, thick, fill=blue!5},
    scale=0.75
]

\path[tube]
    (0,1.45) arc[start angle=90,end angle=270,
                 x radius=.10,y radius=.18]
    -- (7.2,.42) -- (7.2,.66) -- cycle;

\path[tube]
    (0,.85) arc[start angle=90,end angle=270,
                x radius=.10,y radius=.18]
    -- (7.2,.38) -- (7.2,.62) -- cycle;

\path[tube]
    (0,.25) arc[start angle=90,end angle=270,
                x radius=.10,y radius=.18]
    -- (7.2,.34) -- (7.2,.58) -- cycle;

\path[tube]
    (0,-.35) arc[start angle=90,end angle=270,
                 x radius=.10,y radius=.18]
    -- (7.2,.30) -- (7.2,.54) -- cycle;

\fill[gray!55]
    (6.25,.28) -- (7.2,.28) -- (7.2,.72) -- (6.25,.72) -- cycle;

\draw[thick]
    (6.25,.28) -- (7.2,.28) -- (7.2,.72) -- (6.25,.72)--(6.25,.28);

\draw[
    decorate,
    decoration={brace,amplitude=6pt,mirror}
]
    (-.35,1.65) -- (-.35,-.55)
    node[midway,left=9pt] {$N$ tubes};

\draw[
    decorate,
    decoration={brace,amplitude=6pt}
]
    (6.25,1.00) -- (7.2,1.00)
    node[midway,above=9pt]
    {$\dfrac{R^{1/2}}{N}$ squares};

\end{tikzpicture}
    \caption{Incomplete Bush Example}
    \label{fig:BushExample}
\end{figure}
\newpage

\begin{remark}
    The $\ell^p(L^p)$ strip decoupling inequality for the full range $2\le p\le \frac{10}{3}$ follows by the same arguments in this section. The only difference is that it suffices to show a trilinear version of the $\ell^p(L^p)$ strip decoupling, per Demeter's argument in \cite[Theorem 3.1]{demeter2026BeyondCanonical}. This allows us to interpolate between $L^3$ (instead of $L^2$) and $L^4$. We can then reduce to the same incidence estimate Theorem \ref{thm:planarIncidenceKakeya}.
\end{remark}

\section{Reduction to Incidence Estimates for Theorem \ref{thm:strongStripDec}}\label{sec:Reductionl^2(L^3)}\label{sec:p>3}

For the remainder of this paper, we focus on proving Theorem \ref{thm:strongStripDec} for the range $3<p\le \frac{10}{3}$ by modifying the ideas in the previous section to incorporate "two-endedness" of the cubes $Q$ in each tube $T\in \T$. Fix $\epsilon>0$ and $3<p\le \frac{10}{3}$. Let $C(R)$ be the best constant which satisfies \begin{equation}\label{eq:InductionInequality}
    \|F\|_{L^p(B_R)}\le C(R) \left(\sum_{S} \left\|\bS_SF\right\|_{L^p(w_{B_R})}^2\right)^{1/2},
\end{equation} for any $B_R\subset \R^3$ and $F$ with $\supp \wh{F} \subset \mathcal{M}(R^{-1})$. We aim to show $$C(R)\lesssim_\epsilon R^{\epsilon}.$$
We will use $\eta_1, \eta_2, \eta_3$ as small parameters that will be technically useful throughout. In particular, we will show $$C(R)\lessapprox C_{\eta_1, \eta_2,\eta_3} R^{\eta_3}C(R^{1-\eta_2})+R^{O(\eta_1+\eta_2)},$$ and then take $\eta_1 \sim \epsilon$, $\eta_2 = \epsilon^2$, and $\eta_3=\epsilon^3$. As a guide, $\eta_1$ will take the role of $\eta$ in Lemma \ref{label:uniformStats}. The parameters $\eta_2$ and $\eta_3$ will be for the scale we are measuring for "two-endedness" and the quantitative threshold for having "two-endedness," respectively.

\subsection{Wave Packet Decomposition and Pigeonholing}\label{subsect:wavepackets}

Fix a $B_R$ and a function $F$ as above. Recall the wave packet decomposition of $F$ as in Subsection \ref{subsect:wavepackets3}: $$F = \sum_{T\in \T} F_T.$$  By standard pigeonholing, we may assume $\|F_T\|_\infty \sim a$ for each $T\in \T$. We repeat the uniformization step, similar to Lemma \ref{label:uniformStats}, but we will refine $\lambda$ to quantify "two-endedness."  Let $\{B\}$ be $R^{1-\eta_2}$ cubes partitioning $B_R$. For each tube $T\in \T$, let $\mathcal{J}(T)=\{J\}$ be the partition of $T$ by $\{B\}$. That is, each $J=T\cap B$ for some $B$ which means each $J$ is a length $R^{1-\eta_2}$ segment of some $T$.

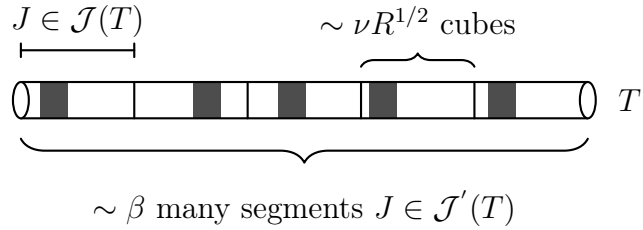
\begin{figure}[b!]
    \centering
    \begin{tikzpicture}[line cap=round, line join=round, scale=0.75]
\begin{scope}[rotate=0]
  \def\L{10}
  \def\h{0.32}

  \foreach \x in {0.35,3.05,4.55,6.15,8.25} {
    \fill[black!70] (\x,-\h) rectangle ++(0.48,0.64);
  }

  \draw[thick] (0,\h) -- (\L,\h);
  \draw[thick] (0,-\h) -- (\L,-\h);
  \foreach \x in {2,4,6,8} {
    \draw[thick] (\x,-\h) -- (\x,\h);
  }

  \draw[thick,fill=white] (0,0) ellipse[x radius=0.14,y radius=\h];
  \draw[thick,fill=white] (\L,0) ellipse[x radius=0.14,y radius=\h];

  \draw[thick] (0,0.88) -- (2,0.88);
  \draw[thick] (0,0.76) -- (0,1.00);
  \draw[thick] (2,0.76) -- (2,1.00);
  \node[above,rotate=0] at (1,0.93) {$J\in \J(T)$};

  \draw[decorate,decoration={brace,amplitude=5pt},thick]
    (6,0.59) -- (8.0,0.59)
    node[midway,above=8pt,rotate=0] {$\sim\nu R^{1/2}\text{ cubes}$};

  \draw[decorate,decoration={brace,mirror,amplitude=9pt},thick]
    (0,-0.68) -- (\L,-0.68)
    node[midway,below=15pt,rotate=0] {$\sim \beta\text{ many segments }J\in \J^{'}(T)$};

  \node[right] at (10.35,0) {$T$};
\end{scope}
\end{tikzpicture}
    \caption{A tube partitioned in segments $J$ with some shading.}
    \label{fig:TubeSegments}
\end{figure}

\newpage

\begin{lemma}\label{label:TwoEndsUniformStats}
   There exists \begin{itemize}
       \item $X\subset B_R$ which is a union of $R^{1/2}$-cubes $\sQ_X$,

       \item a subset of tubes $\T^{'}\subset \T$,

       \item for each tube $T$, a subset $\J^{'}(T)\subset \J(T)$
   \end{itemize} such that $$\|F\|_{L^p(B_R)}\lessapprox \|\sum_{T\in \T^{'}}\sum_{J\in \J^{'}(T)}F_T1_J\|_{L^p(X)}$$ where \begin{enumerate}
        
        \item for each $J\in \J^{'}(T)$, $\#\sQ_X(J)\sim \nu R^{1/2} $
        
        \item for each $T\in \T^{'}$,  $\#\{J \in \J^{'}(T) \}\sim \beta\lesssim R^{\eta_2}$

        \item for each $Q\in \sQ_X$, $\#\{T\in \T^{'} : Q\subset \bigcup_{J\in \J^{'}(T)}J\}\lesssim m$,

        \item $R^{-\eta_1}\#\sQ_X m\lesssim \beta \nu R^{1/2}\#\T^{'}\lesssim  \#\sQ_X m$.
    \end{enumerate}
\end{lemma}

\begin{remark}
    Let $\lambda= \beta \nu$. Moreover, call $m^X(Q)=\#\{T\in \T^{'} : Q\subset \bigcup_{J\in \J^{'}(T)}J\}$.
\end{remark}

\begin{proof}
     We will follow an iterative process to generate $X_j, \T_j, \J_j$, $\nu_j$, $\beta_j$, $m_j$ and stop once we reach the desired incidence condition (4). Let $$F_j= \sum_{T\in \T_j}\sum_{J\in \J_j(T)}F_T1_J.$$ 
     
     \textbf{Initial step}: Let $X_0=B_R$ and $m_0=R$. Apply Lemma \ref{lem:ShortTubeShadingPigeonholing} to create $\T_0\subset \T$ and $\J_0(T)\subset \J(T)$ and dyadic numbers $\nu_0, \beta_0$ so that $$\|F\|_{L^p(X_0)}\lessapprox (\log R)^2 \|F_0\|_{L^p(X_0)}.$$ Note $$ \beta_0 \nu_0 R^{1/2}\#\T_0\sim \sum_{T\in \T_0}\sum_{\substack{Q\in \sQ_{X_0}\\ Q\subset \bigcup_{J\in \J_0(T)}}}1=\sum_{Q\in \sQ_{X_0}}\sum_{\substack{T\in \T_0\\ Q\subset \bigcup_{J\in \J_0(T)}}}1\lesssim \#\sQ_{X_0}m_0 .$$ If also we have\begin{equation}\label{eqn:stoppingCondition}
          \#\sQ_{X_0}m_0\lesssim R^{\eta_1}\beta_0\nu_0R^{1/2}\#\T_0,
     \end{equation} we are done by taking $X=X_0$, $\T^{'}=\T_0$, $\J^{'} = \J_0$, $\nu = \nu_0$, $\beta= \beta_0$, and $m=m_0$. If not, we can repeat the pigeonholing process. 
     
     \textbf{Inductive Step}: We will now construct the next generation. Suppose for some $j\ge 0$, we have constructed $X_{j}, \T_j$, $\J_j$, $\nu_j$, $\beta_j$, $m_{j}$. Pigeonhole the cubes to create $X_{j+1}$ so that for each $Q\in X_{j+1}$, we have $$\#\{T\in \T_j : Q\subset \bigcup_{J\in \J_j(T)}J\}\sim m_{j+1},$$ and $$\|F_j\|_{L^p(X_j)}^p\lesssim (\log R) \|F_j\|_{L^p(X_{j+1})}^p.$$ By Lemma \ref{lem:ShortTubeShadingPigeonholing}, there exists $\T_{j+1}\subset \T_j$ and $\J_{j+1}(T)\subset \J_j(T)$ such that $$\|F_{j}\|_{L^p(X_{j+1})}\lesssim (\log R)^2\|F_{j+1}\|_{L^p(X_{j+1})},$$ and we have for each $J\in \J_{j+1}(T)$, $\#\sQ_{X_{j+1}}(J)\sim \nu_{j+1}R^{1/2}$, and for each $T\in \T_{j+1}$, $\#\J_{j+1}(T)\sim \beta_{j+1}$.
     
     As in the construction of the first generation, we have $$ \beta_{j+1}\nu_{j+1}R^{1/2}\#\T_{j+1}\lesssim \#\sQ_{X_{j+1}}m_{j+1}.$$ If also we have \begin{equation}\label{eqn:stoppingConditionGenereal}        \#\sQ_{X_{j+1}}m_{j+1}\lesssim R^{\eta_1}\beta_{j+1}\nu_{j+1}R^{1/2}\#\T_{j+1},
    \end{equation} we may stop with $X=X_{j+1}$, $\T^{'}=\T_{j+1}$, $\J^{'}=\J_{j+1}$, $\nu=\nu_{j+1}$, $\beta = \beta_{j+1}$, and $m=m_{j+1}$. 
    
    \textbf{Termination}: Every iteration which does not terminate satisfies \begin{align*}
     \beta_{j+1}\nu_{j+1}R^{1/2}\#\T_{j+1} &\lesssim R^{-\eta_1}\#\sQ_{X_{j+1}}m_{j+1}\\
    &\sim R^{-\eta_1}\sum_{Q\in \sQ_{X_{j+1}}}\sum_{\substack{T\in \T_{j}\\ Q\subset \bigcup_{J\in \J_{j}(T)}}}1 \\
    &\lesssim R^{-\eta_1}\sum_{T\in \T_j}\sum_{\substack{Q\in \sQ_{X_j}\\ Q\subset \bigcup_{J\in \J_{j}(T)}}}1 \sim R^{-\eta_1} \beta_{j}\nu_{j}R^{1/2}\#\T_{j}. 
    \end{align*} Therefore, if we iterate $K$ times, $$1\lesssim \beta_{K}\nu_{K}R^{1/2}\#\T_{K} \lesssim R^{-K\eta_1}\beta_{0}\nu_{0}R^{1/2}\#\T_{0}\lesssim R^{O(1)-K\eta_1}.$$ Hence, we need to only iterate at most $K\lesssim \eta_1^{-1}$ times before terminating. Since each iteration gives a factor of $(\log R)^{2+\frac{1}{p}}\lesssim (\log R)^3 $, we get $$\|F\|_{L^p(B_R)}\lesssim (\log R)^{O(\eta_1^{-1})}\|F_{j+1}\|_{L^p(Y)}.$$ We may use the fact that $\log R\lesssim_\eta R^{\eta}$ and take $\eta$ sufficiently small compared to $\eta_1$.
\end{proof}

\begin{lemma}\label{lem:ShortTubeShadingPigeonholing}
   Let $\T$ be a collection of tubes, $\J^{'}(T)\subset \J(T)$, and $X\subset B_R$ a union of $R^{1/2}$-cubes. Then, there exist subsets $\T^{'}\subset \T$ and $\J^{''}(T)\subset\J^{'}(T)$ and dyadic numbers $\beta, \nu$ such that $$\|\sum_{T\in \T}\sum_{J\in J^{'}(T)}F_T1_J\|_{L^p(X)}\lesssim (\log R)^2 \|\sum_{T\in \T^{'}}\sum_{J\in J^{''}(T)}F_T1_J\|_{L^p(X)},$$ and \begin{enumerate}
        \item for each $J\in \J^{''}(T)$, $\#\sQ_X(J)\sim \nu R^{1/2} $,
        
        \item for each $T\in \T^{'}$,  $\#\J^{''}(T)\sim \beta\lesssim R^{\eta_2}$.
    \end{enumerate}  
\end{lemma}

\begin{proof}
    We follow the same pigeonholing of Wang--Wu \cite{wang2024restrictionestimatesusingdecoupling}. Partition $\J^{'}(T)= \bigcup_{\nu}\J^{'}_\nu(T)$ where $\J^{'}_\nu(T) = \{J \in \J^{'}(T) : \#\sQ_X(J)\sim \nu R^{1/2}\}$. For a given $\nu$, partition $\T=\bigcup_{\beta}\T_\beta$ where $\T_\beta = \{T : \#\J^{'}_\nu(T)\sim \beta\}$. Thus, we write \begin{align*}
        \|\sum_{T\in \T}\sum_{J\in J^{'}(T)}F_T1_J\|_{L^p(X)}& = \|\sum_{\nu}\sum_{\beta}\sum_{T\in \T_\beta}\sum_{J\in \J^{'}_\nu(T)}F_T1_J\|_{L^p(X)}\\
        &\lesssim (\log R)^2 \|\sum_{T\in \T_\beta}\sum_{J\in \J^{'}_\nu(T)}F_T1_J\|_{L^p(X)},
    \end{align*} for some pair of dyadics $\beta, \nu$. Let $\T^{'}=\T_\beta$ and $\J^{''}(T)= \J^{'}_\nu(T)$. 
\end{proof}

We will abuse our notation and replace $\T^{'}$ by $\T$ and $\J^{'}$ with $\J$. For each $T\in \T$, define the shading  $$\mathcal{X}(T)= \{Q\in \sQ_X : Q\subset \bigcup_{J\in \J(T)}J\}.$$ From Lemma \ref{label:TwoEndsUniformStats}, $\#\mathcal{X}(T) \sim \lambda R^{1/2}$.  For each $B$ in the covering from Lemma \ref{label:TwoEndsUniformStats}, let $\T_B$ be the collection of tubes $T\in \T$ such that there exists $J\in \J(T)$ with $J\subset B$. Let $$F_B = \sum_{T\in \T_B}F_T.$$ In this notation, Lemma \ref{label:TwoEndsUniformStats} gives $$\|F\|_{L^p(B_R)}^p\lessapprox \sum_B\|F_B\|_{L^p(X\cap B)}^p.$$
We now follow the two-ends dichotomy of  \cite{wang2024restrictionestimatesusingdecoupling}: either $\beta \le R^{\eta_3}$, which is the non-two-ends case, or $\beta\ge R^{\eta_3}$ in which case $$\nu R^{1/2}\sim \#\{Q\in \mathcal{X}(T): Q\subset J\}\sim \beta^{-1}\#\mathcal{X}(T)\le R^{-\eta_3}\#\mathcal{X}(T),$$ for each $J\in \mathcal{J}(T)$. That is, $\mathcal{X}(T)$ is quantitatively two-ends.

\subsection{Non-Two-Ends Scenario}\label{subsec:NonTwoEnds}
We will quickly dispense with the non-two-ends case $\beta \le R^{\eta_3}$. We will need "flat decoupling;" see \cite[Prop. 2.4]{demeter2020smallcapdecouplings} for instance. This is essentially $L^2$ orthogonality interpolated with the $L^\infty$ triangle inequality.

\begin{prop}\label{prop:FlatDec}
    Let $\Omega$ be a rectangular box in $\R^n$ and $\Omega_1, \ldots, \Omega_L$ a partition of $\Omega$ into congruent rectangular boxes that are translates of each other. Then, for $p\ge 2$, $$\|F_\Omega\|_{L^p(\R^n)}\lesssim L^{\frac{1}{2}-\frac{1}{p}}\left(\sum_{i=1}^L\|F_{\Omega_i}\|_{L^p(\R^n)}^2\right)^{1/2}.$$
\end{prop}

We can now establish the non-two-ends case with induction.

\begin{lemma}\label{lem:NonTwoEndsCase}
    $$\left(\sum_B \|F_{B}\|_{L^p( B)}^p\right)^{\frac{1}{p}}\lesssim R^{\eta_3/2} C(R^{1-\eta_2}) \left(\sum_{S}\|\bS_S F\|_{L^p(w_{B_R})}^2\right)^{1/2}.$$
\end{lemma}

\begin{proof}
    Note that since $F_B$ is a sum of wave packets, we still have $$\supp \wh{F_B}\subset \mathcal{M}(O(R^{-1}))\subset\mathcal{M}(R^{\eta_2-1}).$$ So, we apply the inequality \eqref{eq:InductionInequality} at scale $R^{1-\eta_2}$ to get $$\|F_B\|_{L^p(B)}\le C(R^{1-\eta_2})\left(\sum_{S^{'}}\|(F_B)_{S^{'}}\|_{L^p(w_B)}^2\right)^{1/2},$$ where $S^{'}$ are strips with width $(R^{1-\eta_2})^{-1/2}$. We apply flat decoupling Proposition \ref{prop:FlatDec} to get $$\|(F_B)_{S^{'}}\|_{L^p(w_B)}\lesssim (R^{\eta_2/2})^{\frac{1}{2}-\frac{1}{p}}\left(\sum_{S\subset S^{'}}\|(F_B)_S\|_{L^p(w_B)}^2\right)^{1/2}.$$ This gives $$\|F_B\|_{L^p(B)}\lesssim (R^{\eta_2/2})^{\frac{1}{2}-\frac{1}{p}}C(R^{1-\eta_2})\left(\sum_{S}\|(F_B)_S\|_{L^p(w_B)}^2\right)^{1/2}.$$ By Minkowski's inequality, it suffices to see that for each strip $S$, \begin{equation}\label{eq:BallSum}
        (\sum_{B}\|(F_B)_S\|_{L^p(w_B)}^p)^{1/p}\lesssim \beta^{1/2}(R^{-\eta_2})^{\frac{2}{p}-\frac{1}{2}}\|F_S\|_{L^p(w_{B_R})}.
    \end{equation}
    Lemma \ref{lem:NonTwoEndsCase} follows from \eqref{eq:BallSum} since $\beta \le R^{\eta_3}$ and $$\frac{2}{p}-\frac{1}{2}\ge \frac{1}{4}-\frac{1}{2p} \iff p\le \frac{10}{3}.$$ Now, to prove the inequality \eqref{eq:BallSum}, we first crudely bound $$(\sum_{B}\|(F_B)_S\|_{L^p(w_B)}^p)^{1/p}\sim (\sum_{B}\|(F_B)_Sw_B\|_{L^p(w_{B_R})}^p)^{1/p}\le \|(\sum_B |(F_B)_Sw_B|^2 )^{1/2}\|_{L^p(B_R)}.$$ To estimate this term, we use vector-valued interpolation. For $G$ Schwartz, define the operator $$A_S(G)= (A_{S,B}(G))_B= \left(\sum_{\theta\subset S}\sum_{T\in \T_B\cap \T_\theta}G_\theta \psi_Tw_B\right)_B.$$ We will show that $$\|A_S\|_{L^2\to L^2(\ell^2_B)}\lesssim \beta^{1/2}(R^{-\eta_2})^{1/2},\; \|A_S\|_{L^4\to L^4(\ell^2_B)}\lesssim \beta^{1/2}.$$ By complex interpolation, this gives the desired $$\|A_S\|_{L^p\to L^p(\ell^2)} \lesssim \beta^{1/2}(R^{-\eta_2/2})^{\frac{4}{p}-1}= \beta^{1/2}(R^{-\eta_2})^{\frac{2}{p}-\frac{1}{2}}.$$ First, we have by $L^2$ orthogonality, $$\|(\sum_B |A_{S,B}(G)|^2 )^{1/2}\|_{L^2(B_R)}^2 = \sum_B\|A_{S,B}(G)\|_{L^2(w_B)}^2\lesssim \sum_{B}\sum_{\theta\subset S}\sum_{T\in \T_B\cap \T_\theta}\|G_\theta \psi_T\|_{L^2(w_B)}^2.$$ By the locally constant property, we bound $$\|G_\theta\psi_T\|_{L^2(w_B)}^2\lesssim R^{-\eta_2}\|G_\theta \psi_T\|_{L^2(w_{B_R})}^2.$$ Since each $T\in \T$ belongs to $\T_B$ for at most $\beta$ different $B$, we use $L^2$ orthogonality and Plancherel to get $$\sum_{B}\sum_{\theta\subset S}\sum_{T\in \T_B\cap \T_\theta}\|G_\theta \psi_T\|_{L^2(w_B)}^2\lesssim \beta R^{-\eta_2}\|G\|_2,$$ which completes the $L^2$ estimate. Next, since $w_B$ are essentially disjoint, we have $$\|(\sum_B |A_{S,B}(G)|^2)^{1/2}\|_{L^4(B_R)}^4\lesssim \sum_B\|A_{S,B}(G)\|_{L^4(w_B)}^4 \le \sum_B\|A_{S,B}(G)\|_{L^4(w_{B_R})}^4.$$
    
    By the reverse square-function estimate for the parabola, $$\|A_{S,B}(G)\|_{L^4(w_{B_R})}\lesssim \left\|\left(\sum_{\theta\subset S}\sum_{T\in\T_B\cap \T_\theta}|G_\theta \psi_T|^2\right)^{1/2}\right\|_{L^4(w_{B_R})}.$$ Using again the fact that a tube $T$ may only belong to at most $\beta$ different $\T_B$, we have $$\sum_B \|A_{S,B}(G)\|_{L^4(w_{B_R})}^4\lesssim \beta^2 \left\|\left(\sum_{\theta \subset S}|G_\theta|^2\right)^{1/2}\right\|_{L^4(w_{B_R})}^4.$$ By Rubio de Francia square-function estimate and Young's inequality, $$\left\|\left(\sum_{\theta \subset S}|G_\theta|^2\right)^{1/2}\right\|_{L^4(w_{B_R})}\lesssim \|G_S\|_{L^4(w_{B_R})}\lesssim \|G\|_{4},$$ which completes the $L^4$ estimate. 
\end{proof}

\subsection{Two-Ends Scenario}

For each $Q\subset B_R$ and each strip $S$, let $$m^X(Q)= \#\{T\in \T : Q\in \mathcal{X}(T)\},$$ and $$m_S^X(Q)=\#\{T\in \T_S: Q\in \mathcal{X}(T)\}.$$
\begin{prop}\label{prop:LpLHS}
For each $B$, $$\|F_B\|_{L^p(X\cap B)}\lessapprox aR^{\frac{3}{2p}+\eta_1}\lambda^{-\frac{1}{2}+\frac{1}{p}}(\#\sQ_X)^{\frac{1}{p}-\frac{1}{2}}\left(\sum_{S}\sum_{Q\in \sQ_X}m_S^X(Q)\right)^{1/2}.$$    
\end{prop}

\begin{proof}
    From Lemma \ref{label:TwoEndsUniformStats}, $$R^{-\eta_1}m\#\sQ_X\lesssim\lambda R^{1/2}\#\T\lesssim m\#\sQ_X.$$ By H\"{o}lder's inequality, $$\|F_B\|_{L^p(X\cap B)}\le \|F_B\|_{L^2(X\cap B)}^{\frac{4}{p}-1}\|F_B\|_{L^4(X\cap B)}^{2-\frac{4}{p}}.$$
By $L^2$ orthogonality and $L^4$ refined decoupling Theorem \ref{thm:refinedDecoupling}, \begin{align*}
    \|F_B\|_{L^2(X\cap B)}^2 &= \sum_{Q\in \sQ_X, Q\subset B}\|F_B\|_{L^2(Q)}^2\\&\lesssim \sum_{Q\in \sQ_X, Q\subset B}\sum_{T\in \T_B}\|F_T\|_{L^2(Q)}^2\lesssim \#\sQ_X m a^2R^{3/2},
\end{align*}
and $$\|F_B\|_{L^4(X\cap B)}^4 \lessapprox m\sum_{T\in \T_B}\|F_T\|_{L^4(B_R)}^4\lesssim m \#\T a^4R^{2}\lesssim \lambda^{-1}m^2\#\sQ_X a^4R^{3/2}.$$ Therefore, we get \begin{align*}
    \|F_B\|_{L^p(X\cap B)}&\lessapprox (\#\sQ_X m a^2R^{3/2})^{\frac{1}{2}(\frac{4}{p}-1)}(\lambda^{-1}m^2\#\sQ_X a^4R^{3/2})^{\frac{1}{4}(2-\frac{4}{p})}\\
    &= \lambda^{-\frac{1}{2}+\frac{1}{p}}m^{\frac{1}{2}}(\#\sQ_X)^{\frac{1}{p}}aR^{\frac{3}{2p}}.
\end{align*}
Note that $$m\#\sQ_X \lesssim R^{\eta_1}\lambda R^{1/2}\#\T =R^{\eta_1}\sum_{Q\in \sQ_X}m^X(Q)=R^{\eta_1}\sum_{S}\sum_{Q\in \sQ_X}m_S^X(Q).$$ This gives \begin{align*}
    \|F_B\|_{L^p(X\cap B)}&\lessapprox R^{\eta_1}\lambda^{-\frac{1}{2}+\frac{1}{p}}(\#\sQ_X)^{\frac{1}{p}-\frac{1}{2}}aR^{\frac{3}{2p}}\left(\sum_{S}\sum_{Q\in \sQ_X}m_S^X(Q)\right)^{1/2}.
\end{align*}
\end{proof}

\begin{corollary}\label{cor:MultMomentReduction10/3}
    Suppose for each strip $S$,
    \begin{equation}\label{eqn:StripwiseIncidenceMoment10/3}
        \left(\sum_{Q\in \sQ_X}m_S^X(Q)\right)^{\frac{p}{2}}\lessapprox \lambda^{\frac{p}{2}-1}(\#\sQ_X)^{\frac{p}{2}-1}\sum_{Q\subset B_R}m_S(Q)^{p/2}.
    \end{equation}
    Then, $$C(R)\lessapprox C_{\eta_1, \eta_2,\eta_3} R^{\eta_3}C(R^{1-\eta_2})+R^{O(\eta_1+\eta_2)}.$$
\end{corollary}

\begin{proof}
    We apply Proposition \ref{prop:LpLHS} and Lemma \ref{lem:LpRHS} to finish the two-ends case. For the non-two-ends case, we use Lemma \ref{lem:NonTwoEndsCase}.
\end{proof}

Therefore, to prove Theorem \ref{thm:strongStripDec}, it suffices to show \eqref{eqn:StripwiseIncidenceMoment10/3} given the tube-cube statistics from Subsection \ref{subsect:wavepackets}, namely $\#\mathcal{X}(T)\sim \lambda R^{1/2}$ and two-ends for each $T\in \T$.  

\section{Strip Incidence Geometry}\label{sec:StripTubeGeo}
The estimate \eqref{eqn:StripwiseIncidenceMoment10/3} follows from this incidence result. Given Theorem \ref{thm:multMoment}, we can finish the proof of Theorem \ref{thm:strongStripDec}.
\begin{theorem}\label{thm:multMoment}
    Let $\T$ be a collection of tubes in $\R^3$ from a strip $S$ and $X$ the union of $R^{1/2}$-cubes in $B_R$ from a collection $\sQ$ such that $\#\mathcal{X}(T)\sim \lambda R^{1/2}$ and $$\#\{Q\in\mathcal{X}(T): Q\subset J\}\lesssim R^{-\eta_3}\#\mathcal{X}(T)$$ for each $T\in \T$ and $J\in \mathcal{J}(T)$. Then, for each $2< p\le 4$, $$\left(\sum_{Q\in \sQ_X}m_S^X(Q)\right)^{\frac{p}{2}}\lessapprox R^{\eta_2}\lambda^{\frac{3}{2}-\frac{p}{4}}(\#\sQ_X)^{\frac{p}{2}-1}\sum_{Q\subset B_R}m_S(Q)^{p/2}.$$
\end{theorem}

\begin{proof}[Proof of Theorem \ref{thm:strongStripDec}]
    We apply Theorem \ref{thm:multMoment} to Corollary \ref{cor:MultMomentReduction10/3} and note that $$\lambda^{\frac{3}{2}-\frac{p}{4}}\lesssim \lambda^{\frac{p}{2}-1} \iff p\le \frac{10}{3}.$$ Take $\eta_1\sim \epsilon$, $\eta_2=\epsilon^2$ and $\eta_3=\epsilon^3$. Now, by iterating the inequality $$C(R)\lessapprox C_\epsilon R^{\epsilon^3}C(R^{1-\epsilon^2})+R^{\epsilon},$$ we get $C(R)\lesssim_\epsilon R^{\epsilon}.$
\end{proof}

Now, we dedicate the rest of this section to proving Theorem \ref{thm:multMoment}. As in Section \ref{sec:p<3}, we may reduce to the following planar problem. 

\begin{theorem}\label{thm:planarIncidence}
    Let $\T$ be a collection of tubes in $\R^2$ and $\mathcal{X}(T)$ be a shading on each $T\in \T$ such that $\#\mathcal{X}(T)\sim \lambda R^{1/2}$ and $$\#\{Q\in\mathcal{X}(T): Q\subset J\}\lesssim R^{-\eta_3}\#\mathcal{X}(T)$$ for each $T\in \T$ and $J\in \mathcal{J}(T)$. Let $\sQ_X = \bigcup_{T\in \T}\mathcal{X}(T)$. Then, for each $2< p\le 4$, $$\left(\sum_{Q\in \sQ_X}m^X(Q)\right)^{p/2}\lessapprox R^{\eta_2}\lambda^{\frac{3}{2}-\frac{p}{4}}(\#\sQ_X)^{\frac{p}{2}-1}\sum_{Q\subset B_R}m(Q)^{p/2},$$ where $m^X(Q)= \#\{T\in \T: Q\in \mathcal{X}(T)\}$ and $m(Q)= \#\{T\in \T: Q \text{ intersects } T\}.$
\end{theorem}

\begin{proof}[Proof of Theorem \ref{thm:multMoment} from Theorem \ref{thm:planarIncidence}]
    As in Section \ref{sec:p<3}, let $\{U\}$ be parallel plates which capture the tubes of $\T$ coming from a strip $S$. By H\"{o}lder's inequality, we have \begin{align*}
        \left(\sum_{Q\in \sQ_X}m_S^X(Q)\right)^{\frac{p}{2}}  &= \left(\sum_U\sum_{Q\in \sQ_{X\cap U}}m_S^X(Q)\right)^{\frac{p}{2}}\\
        &= \left(\sum_U(\#\sQ_{X\cap U})^{1-\frac{2}{p}}\sum_{Q\in \sQ_{X\cap U}}\frac{m_S^X(Q)}{(\#\sQ_{X\cap U})^{1-\frac{2}{p}}}\right)^{\frac{p}{2}}\\
        &\le \left(\sum_U \#\sQ_{X\cap U}\right)^{\frac{p}{2}-1}\left(\sum_{U}\frac{1}{(\#\sQ_{X\cap U})^{\frac{p}{2}-1}}\left(\sum_{Q\in \sQ_{X\cap U}}m_S^X(Q)\right)^{\frac{p}{2}}\right).
    \end{align*}
    Note that $\#\sQ_{X}= \sum_U \#\sQ_{X\cap U}$ and further, we can apply Theorem \ref{thm:planarIncidence} to the second term above. This immediately gives Theorem \ref{thm:multMoment}.
\end{proof}

We now apply the two-ends Furstenberg estimate to prove Theorem \ref{thm:planarIncidence} and hence, finish the proof of Theorem \ref{thm:multMoment}. We state the two-ends Furstenberg estimate, adapted to our setting and notation.

\begin{definition}
    We say a collection of tubes $\T_0$ in $\R^2$ is Katz-Tao $1$-dimensional with constant $K$ if for every dyadic scale $1\le \rho\lesssim R^{1/2}$ and every $R\times R\rho^{-1}$ fat tube $P$, we have $$\#\{T\in \T_0 : T\subset P\}\le K \frac{R\rho^{-1}}{R^{1/2}}=K\rho^{-1}R^{1/2}.$$
\end{definition}

\begin{theorem}\label{thm:WangWuIncidence}
    (Wang--Wu \cite[Theorem 1.3]{WangWuFurstenberg}) Let $\T_0\subset \T$ be a collection of tubes in $\R^2$ which is Katz-Tao $1$-dimensional with constant $O(1)$, and $X$ a collection of $R^{1/2}$-squares as in Theorem \ref{thm:planarIncidence}. Then, $$\#\T_0 \lessapprox R^{\eta_2/2} R^{-1/2}\lambda^{-3/2}\#\sQ_X.$$ 
\end{theorem}

\begin{proof}[Proof of Theorem \ref{thm:planarIncidence}]

For each dyadic $1\le \rho \lesssim R^{1/2}$, let $\{P_\rho\}$ be a system of $R\times R\rho^{-1}$ tubes covering $B_R$. For each $T\in \T$, define $K_T$ to be the best constant such that for any $\rho$ and any $P_\rho$ with $T\subset P_\rho$, $$M(P_\rho)=\#\{T^{'}\in \T : T^{'}\subset P_\rho\}\le K_T \rho^{-1}R^{1/2}.$$ Pigeonhole to get a dyadic $K$ so that $$\#\T\lessapprox \#\T_{K}$$ and for each $T\in \T_{K}$, $K_T\sim K$. Thus, $\T_{K}$ is Katz-Tao $1$-dimensional with constant $K$. By standard arguments (see for instance, \cite[Lemmas 2.8  and 2.9]{DemeterWang25}), we may partition $\T_{K}$ into roughly $K$ many subcollections of tubes which are Katz-Tao $1$-dimensional with constant $O(1)$ and apply Theorem \ref{thm:WangWuIncidence} to each subcollection. This gives \begin{equation}\label{eq:T_chiI}
    \#\T_{K}\lessapprox K R^{-\frac{1}{2}+\frac{\eta_2}{2}}\lambda^{-\frac{3}{2}}\#\sQ_X.
\end{equation} Alternatively, $$\#\T_{K}K^{\frac{p}{2}-1}\le \sum_{\rho}\sum_{P_\rho}\sum_{T\subset P_\rho}(M(P_\rho)\rho R^{-1/2})^{\frac{p}{2}-1}\le \sum_{\rho}\sum_{P_{\rho}} M(P_{\rho})^{\frac{p}{2}
}\rho^{\frac{p}{2}-1}R^{\frac{1}{2}-\frac{p}{4}}.$$ By the same argument in Theorem \ref{thm:planarIncidenceKakeya}, $$\sum_{P_{\rho}}M(P_\rho)^{\frac{p}{2}}\lesssim R^{\frac{p}{4}-1}\rho^{1-\frac{p}{2}}\sum_{Q\subset B_R}m(Q)^{\frac{p}{2}}.$$ This gives \begin{equation}\label{eq:T_chiII}
    \#\T_{K} \lessapprox K^{1-\frac{p}{2}}R^{-1/2}\sum_{Q\subset B_R}m(Q)^{\frac{p}{2}}.
\end{equation} Together \eqref{eq:T_chiI} and \eqref{eq:T_chiII} give $$(\#\T)^{p/2}\lessapprox R^{\eta_2-\frac{p}{4}}\lambda^{\frac{3}{2}-\frac{3p}{4}}(\#\sQ_X)^{\frac{p}{2}-1}\sum_{Q\subset B_R}m(Q)^{\frac{p}{2}}.$$ Therefore, $$\left(\sum_{Q\in \sQ_X}m^X(Q)\right)^{p/2}\sim \lambda^{p/2}R^{\frac{p}{4}}(\#\T)^{p/2}\lessapprox R^{\eta_2}\lambda^{\frac{3}{2}-\frac{p}{4}}(\#\sQ_X)^{\frac{p}{2}-1}\sum_{Q\subset B_R}m(Q)^{\frac{p}{2}}.$$ 
\end{proof}

\bibliographystyle{plain}
\bibliography{ref}

\end{document}